\documentclass[12pt]{amsart}
\usepackage{amsmath,mathrsfs,amsthm, xypic,verbatim,amscd,color}
\usepackage{txfonts}
\usepackage{graphicx}
\usepackage{colordvi}
\usepackage{times}

\DeclareMathOperator{\Exp}{Exp}

\newtheorem{theorem}{Theorem}[section]
\newtheorem{lemma}[theorem]{Lemma}

\newtheorem{proposition}[theorem]{Proposition}
\newtheorem{corollary}[theorem]{Corollary}

\theoremstyle{definition}
\newtheorem{remark}[theorem]{Remark}

\newtheorem{definition}[theorem]{Definition}

 \newcommand{\ip}[1]{\langle #1 \rangle}
 
  \renewcommand{\tilde}{\widetilde}

  \def\b1{\text{\bf\large 1}}  

  \def\ip<#1>{\langle#1\rangle}   

  \newcommand{\GL}{\operatorname{GL}}
  \newcommand{\PGL}{\operatorname{PGL}}

  \newcommand{\SO}{\operatorname{SO}}

  \newcommand\elal{\Bbb{L}}
    \newcommand\xx{\Bbb{X}}
    \newcommand\exx{\Bbb{X}}
    
 \newcommand\bz{\Bbb{Z}}
\newcommand{\beqn}{\begin{equation}}
\newcommand{\eeqn}{\end{equation}}

\def\u.{^{\bullet}}

\newcommand{\bc}{\mathbb{C}}

\begin{document}

\title{Dimension of zero weight space: an algebro-geometric approach}

\author{Shrawan Kumar and Dipendra Prasad}

\maketitle

\section{Introduction}
Let $G$ be a connected, adjoint,  simple algebraic group
over the complex numbers $\mathbb{C}$ with  a
maximal torus $T$ and a Borel subgroup $B\supset T$. The study of
zero weight spaces in irreducible representations of $G$ has been a
topic of considerable interest; there are many works which study the
zero weight space as a representation space for the Weyl group. In this paper,
we study the variation on the dimension of the zero weight space as the
irreducible representation varies over the set of dominant integral weights for $T$
which are lattice points in  a certain polyhedral cone.

The theorem proved here asserts that the zero weight spaces have dimensions which are
piecewise polynomial  functions on the polyhedral cone of dominant integral weights.
The precise statement of the theorem is given below.

Let $\Lambda= \Lambda (T)$ be the character group of $T$ and
let $\Lambda^{+}\subset \Lambda$ (resp. $\Lambda^{++}$) be the
semigroup of dominant  (resp. dominant regular) weights. Then, by taking derivatives, we can identify
 $\Lambda$ with $Q$, where $Q$ is the root lattice (since $G$ is an adjoint group).
 For $\lambda\in \Lambda^{+}$, let $V(\lambda)$ be the
 irreducible $G$-module with highest weight $\lambda$.
Let $\mu_{0}:\Lambda^{+}\to\mathbb{Z}_{+}$ be the function:
$\mu_{0}=\dim V(\lambda)_{0}$, where $V(\lambda)_{0}$ is the
$0$-weight space of $V(\lambda)$.

Let $\Gamma=\Gamma_G \subset Q$ be the sublattice as in Theorem \ref{thm3.1}.

Also,
let  $\Lambda (\mathbb{R}):= \Lambda\otimes_{\mathbb{Z}}\mathbb{R}$ and let
  $\Lambda^{++} (\mathbb{R})$ be the cone inside
 $\Lambda (\mathbb{R})$ generated by  $\Lambda^{++}$.
Let $C_{1}, \dots, C_{N}\subset \Lambda^{++} (\mathbb{R})$ be the chambers
 (i.e., the GIT
classes  in $\Lambda^{++} (\mathbb{R})$ of
maximal dimension: equal to the dimension
of
$\Lambda(\mathbb{R})$, with respect to the $T$-action) (see Section \ref{sec2}).

For any $w\in W$ and $1\leq i \leq \ell$, define the hyperplane
$$H_{w,i}:=\{\lambda \in \Lambda(\mathbb{R}): \lambda(wx_i)=0\},$$
where $W$ is the Weyl group of $G$ and $\{x_1, \dots, x_\ell\}$ is the basis
of $\mathfrak{t}$ dual to the
basis of $\mathfrak{t}^*$ given by the simple roots. Then, by virtue of
Corollary \ref{chamber}, $C_{1}, \dots, C_{N}$ are the connected components of
$$ \Lambda^{++}(\mathbb{R})\setminus \bigl(\cup_{w\in W, 1\leq i\leq \ell}
\,H_{w,i}\bigr).$$

With this notation, we have the following main result of our paper
(cf. Theorem \ref{thm3.2}).

\begin{theorem}\label{thmA}
 Let
$\overline{\mu}=\mu+\Gamma$ be a coset of $\Gamma$ in $Q$.
 Then, for any GIT class $C_{k}$, $1\leq k \leq N$,
 there exists a polynomial
$f_{\overline{\mu},k}:\Lambda(\mathbb{R})\to\mathbb{R}$ with
rational coefficients of degree $\leq \dim_{\mathbb{C}} X-\ell$, such
that
\begin{equation}
f_{\overline{\mu},k}(\lambda)=\mu_{0}(\lambda),\quad\text{for
  all}\quad \lambda\in \bar{C}_{k}\cap \overline{\mu},
\end{equation}
where  $\overline{C}_{k}$ is the closure of  $C_{k}$ inside $\Lambda(\mathbb{R})$
and $X$ is the full flag variety $G/B$.
 Further, $f_{\Gamma,k}$ has
constant term 1.
\end{theorem}

The proof of  the above theorem relies on Geometric Invariant Theory (GIT).
Specifically,
we realize the function $\mu_0$ restricted to  $\overline{C}_{k}\cap \Lambda$
as an Euler-Poincar\'e characteristic of a reflexive sheaf on a certain
GIT quotient (depending on $C_k$) of  $X=G/B$ via the
maximal torus $T$. Then, one can use the Riemann-Roch theorem for singular varieties to
calculate this Euler-Poincar\'e characteristic. From this calculation, we conclude
 that the function $\mu_0$ restricted to  $\overline{C}_{k}\cap (\mu+\Gamma)$
 is a polynomial function. The result on descent of the homogeneous line bundles
 on $X$ to the GIT quotient plays a crucial role (cf. Lemma \ref{lem3.3}).

We end the paper by determining these piecewise polynomials for the groups of type
$A_2$ and $B_2$ (in Section \ref{sec5}) and $A_3$ (in Section \ref{sec6}),
all of which we
do via some well-known branching laws.

 The results of the paper can easily be extended to show the  piecewise polynomial
behavior of the dimension of any weight space (of a fixed weight $\mu$) in any
finite dimensional irreducible representation $V(\lambda)$.

By a similar proof, we can also obtain a piecewise polynomial
behavior of the dimension of $H$-invariant subspace in any
finite dimensional irreducible representation $V(\lambda)$ of $G$, where
$H\subset G$ is a
reductive subgroup. However, the results in this general case are not as precise
(cf. Remark \ref{newremark}).

It should be mentioned that Meinrenken-Sjamaar [MS] have obtained a result similar to
our above result Theorem \ref{thmA} (also in the generality of $H$-invariants)
by using techniques from Symplectic Geometry.
But, their result in the case of $T$-invariants is less precise than our
Theorem \ref{thmA}.

The example of $\PGL_4$  suggests that the part of  our theorem
describing the domains of validity of these piecewise polynomial functions is
not optimal. Moreover,  our
theorem says nothing about the explicit nature of these polynomials. So this work
should be taken only as a
step towards  the eventual goal  of describing the
variation of the dimension of the zero weight space as the
irreducible representation varies over the set of dominant integral weights for $T$.

\vskip2ex

\noindent
{\bf Acknowledgements:} We are grateful to M. Brion, N. Ressayre, C. Teleman and M. Vergne
 for some very helpful correspondences. We especially thank
Vinay Wagh without whose help the computations on $\GL_4$ given in Section \ref{sec6}
could not have been done.
The first author was supported by the NSF
 grant no. DMS-1201310.
\section{Notation}\label{sec2}

Let $G$ be a connected, adjoint,  semisimple algebraic group
over the complex numbers $\mathbb{C}$. Fix a Borel subgroup $B$ and a
maximal torus $T\subset B$. We denote their Lie algebras by the corresponding
Gothic characters: $\mathfrak{g}$, $\mathfrak{b}$ and  $\mathfrak{t}$
 respectively. Let $R^{+}\subset\mathfrak{t}^{*}$ be the
set of positive roots (i.e., the roots of $B$) and let $\Delta
=\{\alpha_{1},\ldots,\alpha_{\ell}\}\subset R^{+}$ be the set of simple
roots. Let $Q=\bigoplus\limits^{\ell}_{i=1}\mathbb{Z}\alpha_{i}$ be the root
lattice. Then, the group of characters   $\Lambda$ of $T$ can be identified with
 $Q$ (since $G$ is adjoint) by taking the derivative. We will often make this
 identification.  Let $\Lambda^{+}$ (resp.  $\Lambda^{++}$)  be the
semigroup of dominant (resp. dominant regular) weights, i.e.,
$$
\Lambda^{+}:= \{\lambda\in \Lambda:\lambda (\alpha_{i}^{\vee})\in
\mathbb{Z}_{+},\text{~~ for all the simple coroots~ } \alpha^{\vee}_{i}\},
$$
and
$$\Lambda^{++}:= \{\lambda\in \Lambda^+:\lambda (\alpha_{i}^{\vee})\geq 1
\text{~~ for all ~~ } \alpha^{\vee}_{i}\}.$$
Then, $\Lambda^{+}$ bijectively parameterizes the isomorphism classes of finite dimensional
irreducible $G$-modules. For $\lambda\in \Lambda^{+}$, let $V(\lambda)$ be the
 corresponding irreducible $G$-module (with highest weight $\lambda$).

Let $W:=N(T)/T$ be the Weyl group of $G$, where  $N(T)$ is the normalizer of $T$ in $G$.
Let  $\Lambda (\mathbb{R}):= \Lambda\otimes_{\mathbb{Z}}\mathbb{R}$ and let
 $\Lambda^+ (\mathbb{R})$ (resp.  $\Lambda^{++} (\mathbb{R})$) be the cone inside
 $\Lambda (\mathbb{R})$ generated by  $\Lambda^{+}$ (resp.  $\Lambda^{++}$).
Any element $\lambda\in \Lambda (\mathbb{R})$ can uniquely be written as
\begin{equation}\label{eq1}
\lambda=\sum_{i=1}^\ell \,z_{i}\omega_{i},\quad
z_{i}\in \mathbb{R},
\end{equation}
where $\omega_{i}\in \Lambda^{+}(\mathbb{R})$ is the $i$-th fundamental weight:
$$
\omega_{i}(\alpha^{\vee}_{j})=\delta_{i,j}.
$$ Then,
$$
\Lambda^{+}(\mathbb{R})=\oplus_{i=1}^\ell
  \mathbb{R}_{\geq 0}\, \omega_{i},\,\,\,\,\Lambda^{++}(\mathbb{R})=\oplus_{i=1}^\ell
  \mathbb{R}_{>0}\, \omega_{i},
$$
where $\mathbb{R}_{\geq 0}$ (resp.  $\mathbb{R}_{> 0}$)  is the set of  non-negative
 (resp. strictly positive)  real numbers.
We will denote any $\lambda\in \Lambda (\mathbb{R})$ in the coordinates
$z_{\lambda}=(z_{i})_{1\leq i \leq \ell}$ as in \eqref{eq1}.

A function $f:S\subset \Lambda^{+}\to \mathbb{Q}$ defined on a
subset $S$ of $\Lambda^{+}$ is called a {\em polynomial
  function} if there exists a polynomial $\hat{f}(z)\in
\mathbb{Q}[z_{i}]_{1\leq i \leq \ell}$ such that
$f(\lambda)=\hat{f}(z_{\lambda})$, for all $\lambda\in S$.

For any $\lambda\in \Lambda$, we have the $G$-equivariant  line bundle
$\mathcal{L}(\lambda)$ on $X:=G/B$ associated to the principal
$B$-bundle $G\to
G/B$ via the character ${\lambda}^{-1}$ of $B$, i.e.,
$$
\mathcal{L}(\lambda)=G\mathop{\times}^{B}\mathbb{C}_{-\lambda}\to
G/B,
$$
where $\mathbb{C}_{-\lambda}$ denotes the one dimensional $T$-module
with weight $-\lambda$. (Observe that for any $\lambda \in \Lambda$,
the $T$-module structure on $\mathbb{C}_{-\lambda}$ extends to a
$B$-module structure).
The line bundle $\mathcal{L}(\lambda)$ is ample if and only if
$\lambda\in \Lambda^{++}$.

Following Dolgachev-Hu \cite{DH}, $\lambda$, $\mu\in
\Lambda^{++}(\mathbb{R})$ are said to be {\em GIT equivalent} if
$X^{ss}(\lambda)=X^{ss}(\mu)$, where $X^{ss}(\lambda)$
denotes the set of semistable points in $X$ with respect to the
element $\lambda\in \Lambda^{++}(\mathbb{R})$. Recall that if
 $\lambda \in
\Lambda^{++}(\mathbb{Q}):=\oplus_{i=1}^\ell
  \mathbb{Q}_{>0}\, \omega_{i},$ then  $X^{ss}(\lambda)$ is the set of
$T$-semistable points of $X$ with respect to the $T$-equivariant line bundle
$\mathcal{L}(d\lambda)$, for any positive integer $d$ such that $d\lambda\in
 \Lambda^{++}$.

\begin{definition}
By a {\em rational polyhedral cone} $C$ in
$\Lambda^{++}(\mathbb{R})$, one means a subset of
$\Lambda^{++}(\mathbb{R})$ defined by a finite number of linear
inequalities with rational coefficients.
\end{definition}

For a $\mathbb{R}$-linear form $f$ on $\Lambda(\mathbb{R})$ which
is non-negative on $C$, the set of points $c\in C$ such that $f(c)=0$
is called a {\em face} of $C$.

By \cite{DH} or \cite[Proposition 7]{R}, any GIT equivalence class in
$\Lambda^{++}(\mathbb{R})$ is the relative interior of a rational
polyhedral cone in $\Lambda^{++}(\mathbb{R})$ and moreover there
are only finitely many GIT classes (cf. [DH, Theorem 1.3.9] or
[R, Theorem 3]). Let $C_{1}, \dots, C_{N}$ be the GIT
classes of maximal dimension, i.e., of dimension equal to that of
$\Lambda(\mathbb{R})$. These are called {\em chambers}.
Let $X_T(C_k)$ denote the GIT quotient $X^{ss}(\lambda)//T$
for any $\lambda \in C_{k}$.

Since for any $\lambda \in \Lambda^+$, the  irreducible module $V(\lambda)$
has its  zero weight space
$V(\lambda)_{0}$  nonzero, we have
 $X^{ss}(\lambda)\neq \emptyset$ for any $\lambda\in
\Lambda^{++}(\mathbb{R})$.

Let $\mathfrak{t}_{+}:=\{x\in \mathfrak{t}:\alpha_{i}(x)\geq 0$,
for all the simple roots $\alpha_{i}\}$ be the dominant chamber.
Clearly,
\begin{equation}
\mathfrak{t}_{+}=\bigoplus^{\ell}_{i=1}\mathbb{R}_{+}x_{i},\label{eq7}
\end{equation}
where $\{x_{i}\}$ is the basis of $\mathfrak{t}$ dual to the basis of
$\mathfrak{t}^{*}$ consisting of the simple roots, i.e.,
\begin{equation}\label{eq101}
\alpha_{i}(x_{j})=\delta_{i,j}.
\end{equation}

\section{Descent of line bundles to GIT quotients and determination of chambers}
\label{sec3}

There exists the largest lattice $\Gamma\subset Q$ such that for any
 $\lambda\in
\Lambda^{++}\cap \Gamma$, the homogeneous line bundle
$\mathcal{L}(\lambda)$ descends as a line bundle
$\widehat{\mathcal{L}}(\lambda)$ on the GIT quotient
$X_T(\lambda)$.
In fact, $\Gamma$ is determined precisely in \cite[Theorem 3.10]{Ku}
for any simple $G$, which we recall below.

\begin{theorem}\label{thm3.1}
For any simple $G$, $\Gamma=\Gamma_{G}$ is the following lattice
(following the indexing in \cite[Planche I-IX]{B}).
\begin{itemize}
\item[\rm(1)] $G$ of type $A_{\ell} \, (\ell\geq 1):Q$

\item[\rm(2)] $G$ of type $B_{\ell}\, (\ell\geq 3):2Q$

\item[\rm(3)] $G$ of type $C_{\ell}\, (\ell\geq 2):\mathbb{Z}2\alpha_1+\dots
+\mathbb{Z}2\alpha_{\ell -1}+\mathbb{Z}\alpha_{\ell }$

\item[\rm(4)] $G$ of type $D_{4}$:
  $\{n_{1}\alpha_1+2n_{2}\alpha_{2}+n_{3}\alpha_{3}+n_{4}\alpha_{4}:n_{i}\in
  \mathbb{Z}\text{~~ and~~ } n_{1}+n_{3}+n_{4}\text{~~ is even}\}$

\item[\rm(5)] $G$ of type $D_{\ell}\, (\ell\geq 5):$
$$
\{2 n_{1}\alpha_{1}+2n_{2}\alpha_{2}+\cdots+2
n_{\ell-2}\alpha_{\ell-2}+n_{\ell-1}\alpha_{\ell-1}+n_{\ell}\alpha_{\ell}:
 n_{i}\in
\mathbb{Z}\text{~~ and~~ } n_{\ell-1}+n_{\ell}\text{~~ is even}\}
$$

\item[\rm(6)] $G$ of type $G_{2}$: $\mathbb{Z}6
  \alpha_{1}+\mathbb{Z}2\alpha_{2}$

\item[\rm(7)] $G$ of type $F_{4}$:
  $\mathbb{Z}6\alpha_{1}+\mathbb{Z}6\alpha_{2}+\mathbb{Z}12\alpha_{3}+\mathbb{Z}12\alpha_{4}$

\item[\rm(8)] $G$ of type $E_{6}$: $6\tilde{\Lambda}$

\item[\rm(9)] $G$ of type $E_{7}$: $12\tilde{\Lambda}$

\item[\rm(10)] $G$ of type $E_{8}$: $60Q$,
\end{itemize}
where $\tilde{\Lambda}$ is the lattice generated by the fundamental weights.
\end{theorem}

\begin{definition}\label{git} Let $S$ be any connected  reductive algebraic group
acting on a   projective variety  $\exx$ and let  $\elal$ be
an $S$-equivariant line bundle on $\exx$. Let $O(S)$ be the set of all one
parameter
subgroups (for short OPS) in $S$.
 Take any $x\in \exx$ and
 $\delta \in O(S)$. Then,    $\exx$ being projective,  the morphism ${\delta}_x:\Bbb{G}_m\to X$ given by
$t\mapsto \delta(t)x$ extends to a morphism $\tilde{\delta}_x : \Bbb{A}^1\to X$.
Following Mumford, define a number $\mu^{\elal}(x,\delta)$ as follows:
Let $x_o\in X$ be the point  $\tilde{\delta}_x(0)$. Since $x_o$ is
$\Bbb{G}_m$-invariant
via $\delta$, the fiber of  $\elal$ over $x_o$ is a
$\Bbb{G}_m$-module; in particular, it is given by a character of $\Bbb{G}_m$. This
 integer is defined as  $\mu^{\elal}(x,\delta)$.

 Let $V$ be a finite dimensional representation of $S$ and let
$i:\xx\hookrightarrow \Bbb{P}(V)$ be an $S$-equivariant embedding.
Take $\elal:=i^*(\mathcal O(1))$. Let $\lambda \in O(S)$  and let
$\{e_1,\dots,e_n\}$ be a basis of $V$ consisting of eigenvectors,
i.e.,
 $\lambda(t)\cdot e_l=t^{{\lambda}_l}e_l$, for $l=1,\dots,n$. For any $x\in \xx$,
write $i(x)=[\sum_{l=1}^n x_le_l]$. Then, it is easy to see
that, we have ([MFK, Proposition 2.3, page 51])
\begin{equation}\label{muvalue}\mu^{\elal}(x,\lambda)=\max_{l: x_l\neq 0}(-{\lambda}_l).
\end{equation}

\end{definition}

We record the following standard properties of $\mu^{\elal}(x,\delta)$ (cf.
 [MFK, Chap. 2, $\S$1]):
\begin{proposition}\label{propn3.1} For any $x\in \exx$ and $\delta \in O(S)$, we have
 the following (for any $S$-equivariant line bundles
$\elal, \elal_1, \elal_2$):
\begin{enumerate}
\item[(a)]
$\mu^{\elal_1\otimes\elal_2}(x,\delta)=\mu^{\elal_1}(x,\delta)+\mu^{\elal_2}(x,\delta).$
\item[(b)] If $\mu^{\elal}(x,\delta)=0$, then any element of $H^0(\exx,\elal)^S$
which does not vanish at $x$ does not vanish at $\lim_{t\to 0}\delta(t)x$ as well.
\item[(c)] For any projective $S$-variety $\exx'$ together with an $S$-equivariant morphism
$f:\exx'\to \exx$ and any $x'\in \exx'$, we have
$\mu^{f^*\elal}(x',\delta)=\mu^{\elal}(f(x'),\delta).$
\item[(d)] (Hilbert-Mumford criterion) Assume that $\elal$ is ample. Then, $x\in\exx$ is
 semistable (resp. stable)  (with respect to $\elal$) if
and only if $\mu^{\elal}(x,\delta)\geq 0$ (resp.  $\mu^{\elal}(x,\delta)> 0$), for all
non-constant  $\delta\in O(S)$.
\end{enumerate}
\end{proposition}
\begin{lemma}\label{prop3.4}
For any  $\lambda\in
\Lambda^{++}$, the set $X^{s}(\lambda)$ of stable points
(in $X^{ss}(\lambda)$) is nonempty.
\end{lemma}

\begin{proof}
Consider the embedding
$$
i_{\lambda}:X\hookrightarrow \mathbb{P}(V(\lambda)),\quad
gB\mapsto [gv_{\lambda}],
$$
where $v_{\lambda}$ is a highest weight vector in $V(\lambda)$. Then,
the line bundle $\mathcal{O}(1)$ over $\mathbb{P}(V(\lambda))$ restricts to
the  line bundle $\mathcal{L}(\lambda)$ on $X$ via
$i_{\lambda}$ (as can be easily seen).

Consider the open subset $U_{\lambda}\subset X$ defined by
$U_{\lambda}=\{gB\in X:gv_{\lambda}$ has a nonzero component in
each of the weight spaces $V(\lambda)_{w\lambda}$ of weight
$w\lambda$, for all $w\in W\}$.

Since $V(\lambda)$ is an irreducible $G$-module, it is easy to see
that $U_{\lambda}$ is nonempty. We claim that
\begin{equation}
U_{\lambda}\subset X^{s}(\lambda).\label{eq4}
\end{equation}

By the Hilbert-Mumford criterion (cf. Proposition \ref{propn3.1} (d)), it
suffices to prove that for any $gB\in U_{\lambda}$, the Mumford index
\begin{equation}
\mu^{\mathcal{L}(\lambda)}(gB,\sigma)>0,\label{eq5}
\end{equation}
for any nonconstant one parameter subgroup $\sigma:\mathbb{G}_{m}\to
T$. Express
$$
gv_{\lambda}=\sum_{\mu\in X(T)}v_{\mu}\,,
$$
as a sum of weight vectors. Let $\dot{\sigma}$ be the derivative of
$\sigma$ considered as an element of $\mathfrak{t}$. Then, by
the identity \eqref{muvalue},
\begin{align}\label{eq6}
\mu^{\mathcal{L}(\lambda)}(gB,\sigma) &=
\max\limits_{\substack{\mu\in X(T):\\ v_{\mu}\neq
    0}}\, \{-\mu(\dot{\sigma})\}\notag\\
&\geq \max\limits_{w\in W}\, \{\lambda(-w\dot{\sigma})\},\,\,\,\text{since}\,\,
gB\in U_{\lambda}.
\end{align}
Choose  $w'\in W$ such that $-w'\dot{\sigma}\in t_{+}$. Since
$\sigma$ is nonconstant, $-w'\dot{\sigma}\neq 0$.
We next claim that
\begin{equation}
\lambda(-w'\dot{\sigma})>0:\label{eq8}
\end{equation}
To prove this, first observe that any fundamental weight
$\omega_{j}$ belongs to $\oplus^{\ell}_{i=1}\mathbb{Q}_{>0}\alpha_{i}$. (One could
check this case by case for any simple group from \cite[Planche
  I-IX]{B}. Alternatively, one can give a uniform proof as
well.) Thus, by the decomposition \eqref{eq7}, since
$-\omega'\dot{\sigma}\neq 0\in \mathfrak{t}_{+}$, we get
\eqref{eq8}. In particular, by \eqref{eq6},
$\mu^{\mathcal{L}(\lambda)}(gB,\sigma)>0$, proving \eqref{eq5}.
This proves the lemma.
\end{proof}

\begin{proposition}\label{prop3.2}
For $\lambda\in \Lambda^{++}, X^s(\lambda)\neq X^{ss}(\lambda)$ if and only if
there exists $w\in W$ and $x_j$ such that $\lambda(wx_j)=0, $ where $x_i\in \mathfrak{t}$ is defined by
\eqref{eq101}.
\end{proposition}
\begin{proof} Assume first that $X^s(\lambda)\neq X^{ss}(\lambda)$. Take $x\in
 X^{ss}(\lambda)\setminus X^s(\lambda)$. Then, by the
Mumford criterion Proposition \ref{propn3.1} (d), there exists a non-constant one
parameter subgroup $\delta$ in $T$ such that $\mu^{\mathcal{L}(\lambda)}(x, \delta)
=0$. Since both of  $X^s(\lambda)$ and $X^{ss}(\lambda)$ are $N(T)$-stable under
 the left multiplication on $X$ by $N(T)$ (by loc. cit.), we
can assume that $\delta$ is $G$-dominant, i.e.,  the derivative $\dot{\delta}\in \mathfrak{t}_+$.
Thus, by Proposition \ref{propn3.1} (b),  $x_o:=\lim_{t\to 0}\delta(t)x \in
X^{ss}(\lambda)$, since $x$ is semistable. Let $G^\delta$ be the fixed point subgroup of $G$ under
the conjugation action by $\delta$. Then,  $G^\delta$ is a (connected) Levi subgroup of $G$. Let
$W^{G^\delta}$ be the set of minimal length coset  representatives in the cosets $W/ W_{G^\delta}$,
where $ W_{G^\delta} \subset W$ is the Weyl group of $G^\delta$. The fixed point set
of $X$ under the left multiplication by $\delta$ is given by $X^\delta= \sqcup_{v\in W^{G^\delta}}\,
 G^\delta v^{-1}B/B$. Let $w\in W^{G^\delta}$ be such that  $x_o\in G^\delta w^{-1}B/B$. Thus, by [Ku, Lemma 3.4],
 \beqn \label{eqn3.2} w^{-1}\lambda\in \sum_{\alpha_i\in \Delta(G^\delta)}\,\bz
 \alpha_i,
 \eeqn
where $ \Delta(G^\delta)\subset  \Delta $ is the set of simple roots of $G^\delta$.
 Since $\delta$ is non-constant, $G^\delta$ is a proper Levi subgroup.
 Take $\alpha_j\in \Delta\setminus \Delta(G^\delta)$. Then, by \eqref{eqn3.2},
 $\lambda (wx_j)=0.$

 Conversely, assume that
 \beqn\label{eqn3.3}\lambda(wx_j)=0,
 \eeqn
 for some $w\in W$ and some $x_j$.
For any $1\leq i\leq \ell$, let $L_i$ be the Levi subgroup containing $T$
such that $\Delta(L_i)=\Delta \setminus \{\alpha_i\}$. By the assumption
\eqref{eqn3.3}, $w^{-1}\lambda \in \sum_{\alpha_i\in  \Delta(L_j)}\,\bz
\alpha_i.$ Moreover, we can choose $w\in W^{L_j}$ and hence  $w^{-1}\lambda$ is
a dominant weight for $L_j$. In particular,  $v_{w^{-1}\lambda}$ is a highest weight vector
for $L_j$, where   $v_{w^{-1}\lambda}$ is a nonzero vector of (extremal) weight
  $w^{-1}\lambda$ in $V(\lambda)$. (To prove this, observe that $|  w^{-1}\lambda+\alpha_i|>
 |\lambda| $ for any $\alpha_i\in \Delta(L_j)$, and hence  $ w^{-1}\lambda+\alpha_i$
can not be a weight of $V(\lambda)$.) Thus, the $L_j$-submodule $V_{L_j}(w^{-1}\lambda)$ of
$V(\lambda)$ generated by $v_{w^{-1}\lambda}$ is an irreducible $L_j$-module.   By [Ku, Lemma 3.1],
applied to the  $L_j$-module $V_{L_j}(w^{-1}\lambda)$, we get that  $V_{L_j}(w^{-1}\lambda)$
 contains the zero weight space.
Hence, by [Ku, Lemma 3.4], there exists a $g\in L_j$ such that $gw^{-1}B\in
X^{ss}(\lambda)$. Define the one parameter subgroup  $\delta_j := \Exp(zx_j)$.
Then, $\mu^{\mathcal{L}(\lambda)}(gw^{-1}B, \delta_j)=
\mu^{\mathcal{L}(\lambda)}(w^{-1}B, \delta_j)$, since $g$ fixes $\delta_j$.
But, $\mu^{\mathcal{L}(\lambda)}(w^{-1}B, \delta_j)=0$, by \eqref{muvalue} (due to the assumption \eqref{eqn3.3}). Thus,
$gw^{-1}B \notin X^s(\lambda)$ by Proposition \ref{propn3.1} (d).
\end{proof}

For any $w\in W$ and $1\leq i \leq \ell$, define the hyperplane
$$H_{w,i}:=\{\lambda \in \Lambda(\mathbb{R}): \lambda(wx_i)=0\}.$$
Decompose into connected components:
$$ \Lambda^{++}(\mathbb{R})\setminus \bigl(\cup_{w\in W, 1\leq i\leq \ell}\,H_{w,i}\bigr)= \sqcup_{k=1}^N\, C_k.$$
The following corollary follows immediately from Proposition \ref{prop3.2} and [DH, Theorems 3.3.2 and 3.4.2].
\begin{corollary} \label{chamber} With the notation as above, $\{C_1, \dots, C_N\}$ are precisely the GIT classes of maximal dimension
(equal to dim $\mathfrak{t}$).
\end{corollary}
\begin{lemma}\label{lem3.3}
For any GIT class $C_{k}$ (of maximal dimension) and any $\lambda\in
\Gamma$, the line bundle $\mathcal{L}(\lambda)$ descends as a
line bundle on the GIT quotient $X_T(C_k)$. We denote this
line bundle by $\widehat{\mathcal{L}}_{C_{k}}(\lambda)$.
\end{lemma}

\begin{proof}
By Theorem \ref{thm3.1}, for any $\lambda\in \Lambda^{++}\cap
\Gamma$, the line bundle $\mathcal{L}(\lambda)$ on $X$ descends
to a line bundle on $X_T(\lambda)$. Hence, for any
$\lambda\in \Gamma\cap C_{k}$, the line bundle
$\mathcal{L}(\lambda)$ descends to a line bundle
$\widehat{\mathcal{L}}_{C_{k}}(\lambda)$ on $X_T(C_k)$.

Let $\mathbb{Z}(\Gamma\cap C_{k})$ denote the subgroup of $\Gamma$
generated by the semigroup $\Gamma\cap C_{k}$. For any
$\lambda=\lambda_{1}-\lambda_{2}\in \mathbb{Z}(\Gamma\cap C_{k})$
(for $\lambda_{1}$, $\lambda_{2}\in \Gamma\cap C_{k}$), define
$$
\widehat{\mathcal{L}}_{C_{k}}(\lambda)=\widehat{\mathcal{L}}_{C_{k}}(\lambda_{1})\otimes
\widehat{\mathcal{L}}_{C_{k}}(\lambda_{2})^{*}.
$$

We now show that $\widehat{\mathcal{L}}_{C_{k}}(\lambda)$ is well
defined, i.e., it does not depend upon the choice of the decomposition
$\lambda=\lambda_{1}-\lambda_{2}$ as above. Take another decomposition
$\lambda=\lambda'_{1}-\lambda'_{2}$, with $\lambda'_{1}$,
$\lambda'_{2}\in \Gamma\cap C_{k}$. Thus,
$\lambda_{1}+\lambda'_{2}=\lambda'_{1}+\lambda_{2}\in \Gamma \cap
C_{k}$ (since $\Gamma\cap C_{k}$ is a semigroup). In
particular,
$\widehat{\mathcal{L}}_{\mathbb{C}_{k}}(\lambda_{1}+\lambda'_{2})\simeq
\widehat{\mathcal{L}}_{C_{k}}(\lambda'_{1}+\lambda_{2})$.

But, from the uniqueness of $\widehat{L}_{C_{k}}(\lambda)$
(cf. \cite[\S\ 3]{T}), we have
$\widehat{\mathcal{L}}_{C_{k}}(\lambda_{1}+\lambda'_{2})\simeq
\widehat{\mathcal{L}}_{C_{k}}(\lambda'_{1})\otimes
\widehat{\mathcal{L}}_{C_{k}}(\lambda_{2})$. This proves the assertion
that $\mathcal{L}_{C_{k}}(\lambda)$ is well defined.

Observe that, by definition,  $C_{k}$ is  an open convex cone in
 $\Lambda(\mathbb{R})$.
We next claim that
\begin{equation}\label{eq102}
\mathbb{Z}(\Gamma\cap C_{k})=\Gamma.
\end{equation}

Take a $\mathbb{Z}$-basis $\{\gamma_{1},\ldots,\gamma_{\ell}\}$ of $\Gamma$ and
let $d:=\max_{i}||\gamma_{i}||$, with respect to a norm $||\cdot||$ on
$\Lambda(\mathbb{R})$. Take a `large enough' $\gamma\in \Gamma\cap
{C}_{k}$ such that the closed ball $B(\gamma,d)$ of radius
$d$ centered at $\gamma$ is contained in ${C}_{k}$. Then,
for any $1\leq i\leq \ell$, $\gamma+\gamma_{i}\in B(\gamma,d)$ and hence $\gamma$,
$\gamma+\gamma_{i}\in \Gamma\cap C_{k}$ for any $i$. Thus, each $\gamma_{i}\in
\mathbb{Z}(\Gamma\cap C_{k})$ and hence
$\Gamma=\mathbb{Z}(\Gamma\cap C_{k})$, proving the assertion \eqref{eq102}.
Thus, the lemma is proved.
\end{proof}

\section{The main result and its proof}\label{sec4}

Let $\mu_{0}:\Lambda^{+}\to\mathbb{Z}_{+}$ be the function:
$\mu_{0}=\dim V(\lambda)_{0}$, where $V(\lambda)_{0}$ is the
$0$-weight space of $V(\lambda)$.
Following the notation from Sections \ref{sec2} and \ref{sec3}, the following is  our main result.

\begin{theorem}\label{thm3.2}
Let $G$ be a connected, adjoint, simple algebraic group. Let
$\overline{\mu}=\mu+\Gamma$ be a coset of $\Gamma$ in $Q$, where
$\Gamma$ is as in Theorem \ref{thm3.1}.
 Then, for any GIT class $C_{k}$ (of maximal dimension), $1\leq k \leq N$,
 there exists a polynomial
$f_{\overline{\mu},k}:\Lambda(\mathbb{R})\to\mathbb{R}$ with
rational coefficients of degree $\leq \dim_{\mathbb{C}} X-\ell$, such
that
\begin{equation}\label{eq2}
f_{\overline{\mu},k}(\lambda)=\mu_{0}(\lambda),\quad\text{for
  all}\quad \lambda\in \bar{C}_{k}\cap \overline{\mu},
\end{equation}
where  $\overline{C}_{k}$ is the closure of  $C_{k}$ inside $\Lambda(\mathbb{R})$.
 Further, $f_{\Gamma,k}$ has
constant term 1.
\end{theorem}
\begin{proof}
By the Borel-Weil theorem, for any $\lambda\in \Lambda^+$,
\begin{equation}\label{eq103}
\mu_{0}(\lambda)=\dim
\left({H}^0(X,\mathcal{L}(\lambda))^{T}\right),
\end{equation}
since
$$
\dim (V(\lambda)_{0})=\dim \left((V(\lambda)^{*})_{0}\right).
$$
Moreover, by the Borel-Weil-Bott theorem, for $\lambda\in
\Lambda^{+}$,
\begin{equation}
H^{p}(X,\mathcal{L}(\lambda))=0,\quad\text{for all}\quad p>0.\label{eq10}
\end{equation}

We first prove the theorem for $ \lambda\in {C}_{k}\cap \overline{\mu}$:

Take $\lambda\in C_{k}\cap \overline{\mu}$. Let
$\pi:X^{ss}(C_{k})\to X_T(C_k)$ be the standard
quotient map. For any $T$-equivariant sheaf $\mathcal{S}$ on
$X^{ss}(C_{k})$, define the $T$-invariant direct image sheaf
$\pi^{T}_{*}(\mathcal{S})$ as the sheaf on $X_T(C_k)$
with sections $U\mapsto
\Gamma({\pi}^{-1}(U),\mathcal{S})^{T}$. Then, by Lemma \ref{lem3.3},
and the projection formula for $\pi^{T}_{*}$,
\begin{equation}
\pi^{T}_{*}(\mathcal{L}(\lambda))\simeq
\pi^{T}_{*}(\mathcal{L}(\mu))\otimes
\widehat{\mathcal{L}}_{C_{k}}(\lambda-\mu). \label{eq11}
\end{equation}
By \cite[Remark 3.3(i)]{T} and \eqref{eq10}, we get
\begin{align}
H^{p}\left(X_T(C_k), \
\pi^{T}_{*}(\mathcal{L}(\lambda))\right) &\simeq
H^{0}(X,\mathcal{L}(\lambda))^{T},\quad\text{for}\quad
p=0 \notag\\
 &= 0,\quad\text{otherwise}.\label{eq12}
\end{align}
Thus, for $\lambda\in C_{k}\cap \overline{\mu}$, by \eqref{eq103},
\begin{equation}
\mu_{0}(\lambda)=\chi\left(X_T(C_k),
\ \pi^{T}_{*}(\mathcal{L}(\lambda))\right),\label{eq13}
\end{equation}
where for any projective variety $Y$ and a coherent sheaf
$\mathcal{S}$ on $Y$, we define the Euler-Poincar\'e characteristic
$$
\chi(Y,\mathcal{S}):=\sum_{i\geq 0}(-1)^{i}\dim H^{i}(Y,\mathcal{S}).
$$
Now, take a basis (as a $\mathbb{Z}$-module) $\{\gamma_{1},\ldots,\gamma_{\ell}\}$
of the lattice $\Gamma\subset \Lambda(\mathbb{R})$.
Then, for any $\lambda=\mu+\sum^{\ell}_{i=1}a_{i}\gamma_{i}\in
\bar{\mu}$, with $a_{i}\in \mathbb{Z}$, we have by \eqref{eq11},
\begin{equation}
\pi^{T}_{*}(\mathcal{L}(\lambda))\simeq
\pi^{T}_{*}(\mathcal{L}(\mu))\otimes
\widehat{\mathcal{L}}_{C_{k}}(\sum a_{i}\gamma_{i}).\label{eq14}
\end{equation}
Thus, by the Riemann-Roch theorem for singular varieties
(cf. \cite[Theorem 18.3]{F}) applied to the sheaf
$\pi^{T}_{*}(\mathcal{L}(\lambda))$, we get for any $\lambda
=\mu+\sum a_i\gamma_i\in
\bar{\mu}$,
\begin{align}
& \chi
\left(X_T(C_k),\pi^{T}_{*}(\mathcal{L}(\lambda))\right)=\notag\\
&\qquad \sum_{n\geq
  0}\int\limits_{X_T(C_k)}\frac{(a_{1}c_{1}(\gamma_{1})+\cdots+a_{\ell}c_{1}(\gamma_{\ell}))^{n}}{n!}\cap
\tau \left(\pi^{T}_{*}(\mathcal{L}(\mu))\right), \label{eq15}
\end{align}
where $\tau(\pi^{T}_{*}(\mathcal{L}(\mu)))$ is a certain class in
the chow group
$A_{*}(X_T(C_k))\otimes_{\mathbb{Z}}\mathbb{Q}$ and
$c_{1}(\gamma_{i})$ is the first Chern class of the line bundle
$\widehat{\mathcal{L}}_{C_{k}}(\gamma_{i})$.
Combining \eqref{eq13} and \eqref{eq15}, we get that for any
$\lambda\in C_{k}\cap \bar{\mu}$, $\mu_{0}(\lambda)$
is a polynomial $f_{\bar{\mu}, k}$ with rational coefficients  in the variables $\{a_{i}\}:\lambda=\mu
+\sum\limits^{\ell}_{i=1}a_{i}\gamma_{i}$.

Since $X^{s}(C_{k})\neq \emptyset$ by Lemma
\ref{prop3.4}, $\dim \,(X_T(C_k))=\dim X-\ell$. Thus, deg
 $f_{\bar{\mu}, k}\leq\dim X-\ell. $
This proves the theorem for $\lambda\in C_{k}\cap
\overline{\mu}$.

 We now come to the proof of the theorem for any $\lambda \in
 \bar{C}_k\cap\bar{\mu}$:

 Let $P=P_\lambda\supset B$ be the unique parabolic subgroup such that
 the line bundle $\mathcal{L}(\lambda)$ descends as an ample line bundle
 (denoted $\mathcal{L}^P(\lambda)$) on $X^P:=G/P$ via the standard
 projection $q:G/B\to G/P$. Fix $\mu\in C_k\cap \Lambda$. By [T, $\S$1.2],
 applied to $q:G/B\to G/P$, we get that $q^*(\mathcal{L}^P(\lambda))$ is adapted
 to the stratification on $X$ induced from $q^*(\mathcal{L}^P(\lambda))+\epsilon
 \mathcal{L}(\lambda)$, for any small rational $\epsilon >0$ (cf. loc. cit.
 for the terminology). Thus, by [T, Theorem 3.2.a and Remarks 3.3], we get that
 (for any $\lambda \in \bar{C}_k\cap\bar{\mu}$)
 \begin{equation}\label{eq105}
 \mu_{0}(\lambda)=\chi\left(X_T(C_k),
\ \pi^{T}_{*}q^*(\mathcal{L}^P(\lambda))\right)=\chi\left(X_T(C_k),
\ \pi^{T}_{*}(\mathcal{L}(\lambda))\right).
\end{equation}
Hence, the identity \eqref{eq13} is established for any $\lambda \in
\bar{C}_k\cap\bar{\mu}$. Thus, by the above proof, $\mu_0(\lambda)=f_{\bar{\mu},k}$, where
$f_{\bar{\mu},k}$ is the polynomial given above.

By the formula \eqref{eq15}, the constant term of
$f_{\Gamma,k}$ is equal to
$$
\chi\left(X_T(C_k), \ \pi^{T}_{*}(\mathcal{L}(0))\right),
$$
which is 1 by the identity \eqref{eq105}, since $\mu_0(0)=1$.
This completes the proof of the theorem.
\end{proof}
\begin{remark} \label{newremark} (a) By a similar proof, we can obtain a piecewise polynomial
behavior of the dimension of any weight space (of a fixed weight $\mu$) in any
finite dimensional irreducible representation $V(\lambda)$, by considering the
GIT theory associated to the $T$-equivariant line bundle $\mathcal{L}(\lambda)$
 twisted by the character $\mu^{-1}$.

 (b) By a similar proof, we can also obtain a piecewise polynomial
behavior of the dimension of $H$-invariant subspace in any
finite dimensional irreducible representation $V(\lambda)$ of $G$, where
$H\subset G$ is a
reductive subgroup. In this case, we will need to apply the
GIT theory to the line bundle $\mathcal{L}(\lambda)$ itself but with respect to
the group $H$. However, in this general case, we do not have a precise
description of the lattice $\Gamma$ as in Theorem \ref{thm3.1}, nor do we have an
explicit description of the GIT classes of maximal dimension as in Corollary
\ref{chamber}.

(c) As pointed out by Kapil Paranjape, we can obtain the polynomial behaviour of $\chi
\left(X_T(C_k),\pi^{T}_{*}(\mathcal{L}(\lambda))\right)$ as in the above
proof (by using the Riemann-Roch theorem) more simply by applying Snapper's
theorem (cf. [K, Theorem in Section 1]). However, the use of Riemann-Roch theorem
gives a more precise result.
\end{remark}

\section{ Examples of $A_2$ and $B_2$} \label{sec5}

In this section, we calculate the dimension of the $T$-invariant subspace in an irreducible
representation of the  rank 2 groups $G$ of types $A_2$ and $B_2$. In these cases, we can do the calculation via
certain well-known branching laws to certain subgroups. But
lacking any such general branching laws, we have not been able to handle $G_2$.

We recall that
irreducible representations of $\GL_{n+1}(\bc)$ are parametrized by their
highest weights,
which is an $(n+1)$-tuple of integers:

$$\lambda_1 \geq \lambda_2 \geq \cdots \geq  \lambda_n \geq  \lambda_{n+1}.$$

It is a well-known theorem that an irreducible representation of $\GL_{n+1}(\bc)$ when restricted
to $\GL_n(\bc)$ decomposes as a sum of irreducible representations with
highest weights
$(\mu_1 \geq  \mu_2  \geq \cdots  \geq \mu_n)$ with
$$\lambda_1 \geq \mu_1 \geq \lambda_2 \geq \mu_2 \geq \cdots \geq  \lambda_n \geq \mu_n \geq
\lambda_{n+1},$$
and that these representations of $\GL_n(\bc)$ appear with multiplicity exactly
one (cf. [GW, Theorem 8.1.1]).

Note that for an irreducible representation of $\GL_{n+1}(\bc)$ to have a nonzero zero weight space, it is necessary (and sufficient)
for it to have trivial central character. For determining the zero weight space
of a  representation of $\GL_{n+1}(\bc)$ with trivial central character,  it suffices to restrict it to
$\GL_n(\bc)$ and  consider those summands which have zero weight spaces for $\GL_n(\bc)$,
and then to add these zero weight spaces of $\GL_n(\bc)$.

We  calculate the dimension of the zero weight space of an  irreducible
representation of $\GL_3(\bc)$ by restricting
the representation to $\GL_2(\bc)$, and noting that an irreducible representation of $\GL_2(\bc)$ parametrized
by $(\mu_1 \geq \mu_2)$ has a nonzero weight space if and only if $\mu_1+\mu_2 =0$
 and, in this case, the dimension
of the zero weight space is 1.
With these preliminaries, we leave the details of the straightforward proof of the following lemma to the reader.

\begin{lemma} An irreducible representation of
$\GL_{3}(\bc)$ with highest weight
$(\lambda_1 \geq \lambda_2 \geq  \lambda_3),$
and with trivial central character, i.e., $\lambda_1 +  \lambda_2 +  \lambda_3 = 0$,
has zero weight space of dimension

\begin{enumerate}

\item  $\lambda_1-\lambda_2+1$, if  $\lambda_2 \geq 0$, and

\item  $\lambda_2-\lambda_3+1$, if  $\lambda_2 \leq 0$.

\end{enumerate}
\end{lemma}

We next recall that irreducible representations of $\SO_{2n+1}(\bc)$ are
parametrized by their highest weights,
which is an $n$-tuple of integers with
$$\lambda_1 \geq \lambda_2 \geq \cdots \geq  \lambda_n \geq 0 .$$
Similarly, the  irreducible representations of $\SO_{2n}(\bc)$ are parametrized
by their highest weights,
which is an $n$-tuple of integers with
$$\lambda_1 \geq \lambda_2 \geq \cdots \geq  |\lambda_n| .$$
It is a well-known theorem that an irreducible representation of $\SO_{2n+1}(\bc)$ when restricted
to $\SO_{2n}(\bc)$ decomposes as a sum of irreducible representations with highest
weights
$(\mu_1 \geq  \mu_2  \geq \cdots  \geq |\mu_n|)$ with
$$\lambda_1 \geq \mu_1 \geq \lambda_2 \geq \mu_2 \geq \cdots \geq  \lambda_n \geq |\mu_n|,$$
and that these representations of $\SO_{2n}(\bc)$ appear with multiplicity
exactly one (cf. [GW, Theorem 8.1.3]).

We use this branching law from $\SO_5(\bc)$ to $\SO_4(\bc)$ to calculate the
dimension of the
zero weight space in an irreducible representation of $\SO_5(\bc)$. For this,
we again note that the
zero weight space in a $\SO_5(\bc)$-representation  is captured by those
subrepresentations of
$\SO_4(\bc)$ which have nonzero zero weight space. Further, note that $\SO_4(\bc)$ being the quotient
of ${\rm SU}_2(\bc) \times {\rm SU}_2(\bc)$ by the diagonal central element $\pm 1$, an irreducible representation of
$\SO_4(\bc)$ has nonzero zero weight space if and only if its central character is trivial, and in this case
the zero weight space is 1 dimensional.
We also need to  use the fact that
the  irreducible representation of $\SO_{2n}(\bc)$  parametrized by
$(\lambda_1, \lambda_2, \cdots , \lambda_n)$ has trivial central character if and
only if $\lambda_1+ \lambda_2+  \cdots +\lambda_n$ is an even integer.

With these preliminaries, we leave the details of the straightforward proof of the following lemma to the reader.

\begin{lemma} An irreducible representation of
$\SO_{5}(\bc)$ with highest weight
$\lambda_1 \geq \lambda_2 \geq  0)$
has zero weight space of dimension

\begin{enumerate}

\item  $(\lambda_1-\lambda_2)\cdot \lambda_2 + \displaystyle{\frac{\lambda_1 +
\lambda_2}{2}} + 1$,
if  $\lambda_1 + \lambda_2$ is an even integer.

\item  $(\lambda_1-\lambda_2)\cdot \lambda_2 +
\displaystyle{\frac{\lambda_1 + \lambda_2}{2}} + \frac{1}{2}$,
 if  $\lambda_1 + \lambda_2$ is an odd integer.

\end{enumerate}
\end{lemma}

\section{The Example of  $\PGL_4$} \label{sec6}

In this section, we compute the dimension of the zero weight space of any irreducible representation of  $G=\PGL_4(\bc)$.

\begin{theorem}
For an irreducible  representation of $\GL_4(\bc)$ with highest weight
$(\lambda_1 \geq \lambda_2 \geq   \lambda_3 \geq  \lambda_4),$
and with trivial central character, i.e.,
$\lambda_1 +  \lambda_2 +   \lambda_3 +  \lambda_4 = 0,$ the dimension
$d(\lambda_1,\lambda_2,\lambda_3,\lambda_4)$ of the zero weight space is given as a piecewise
polynomial in the domain $\lambda_1 \geq \lambda_2 \geq   \lambda_3 \geq  \lambda_4$ as follows:

\begin{enumerate}

\item $\lambda_2 \leq 0$, where it is given by the polynomial
$$p_1(\lambda_1,\lambda_2,\lambda_3,\lambda_4)=
\frac{1}{2} (\lambda_2-\lambda_3+1)(\lambda_3-\lambda_4+1)
(\lambda_2-\lambda_4+2).$$

\item $\lambda_3 \geq 0$, where it is given by the polynomial
$$p_2(\lambda_1,\lambda_2,\lambda_3,\lambda_4)= \frac{1}{2} (\lambda_1-\lambda_2+1)(\lambda_2-\lambda_3+1) (\lambda_1-\lambda_3+2).$$

\item $\lambda_2 > 0, \lambda_3 <0, \lambda_1 + \lambda_4 \geq 0$, where it
is given by the polynomial
\begin{equation*}
p_3(\lambda_1,\lambda_2,\lambda_3,\lambda_4) =  -\frac{1}{2}
(\lambda_1+\lambda_2+2\lambda_3+1)(-\lambda_1\lambda_2+2\lambda_2^2+
\lambda_1\lambda_3+\lambda_2\lambda_3+\lambda_3^2-\lambda_1+\lambda_3-2).
\end{equation*}

\item $\lambda_2 > 0, \lambda_3 <0, \lambda_1 + \lambda_4 \leq 0$, where it is
given by the polynomial
\begin{eqnarray*}
p_4(\lambda_1,\lambda_2,\lambda_3,\lambda_4) = \frac{1}{2} (-\lambda_1+\lambda_2-1)(-\lambda_1\lambda_2+\lambda_1\lambda_3+\lambda_2\lambda_3+3\lambda_3^2-
\lambda_1-2\lambda_2-\lambda_3-2).
 \end{eqnarray*}

\end{enumerate}

The automorphism $(\lambda_1, \lambda_2,\lambda_3,\lambda_4)\longrightarrow
(-\lambda_4, -\lambda_3, -\lambda_2,-\lambda_1)$, which corresponds to
taking a representation to its dual, interchanges the regions (1) and (2), and their polynomials, and similarly regions (3) and (4) and their
polynomials. Further, we have
$$p_3-p_4 = (\lambda_2+\lambda_3) -(\lambda_2+\lambda_3)^3.$$
\end{theorem}
\begin{proof}
The method we follow
to prove this theorem  is also  based on the restriction of a $\GL_4(\bc)$ representation
to $\GL_3(\bc)$, as we did in the previous
section for the calculation of the zero weight space for
$\GL_3(\bc)$-representations. We start with an
irreducible representation of $\GL_4(\bc)$ with highest weight
$(\lambda_1 \geq \lambda_2 \geq   \lambda_3 \geq  \lambda_4),$
and with trivial central character, i.e.,
$\lambda_1 +  \lambda_2 +   \lambda_3 +  \lambda_4 = 0.$

We  look at irreducible representations of $\GL_3(\bc)$ with highest weight
$(\mu_1 \geq  \mu_2  \geq \mu_3)$  appearing in this representation of
$\GL_4(\bc)$. Thus, we have
$$\lambda_1 \geq \mu_1 \geq \lambda_2 \geq \mu_2 \geq   \lambda_3 \geq \mu_3 \geq
\lambda_4.$$
For analyzing the zero weight space, it suffices to consider only those representations
of $\GL_3(\bc)$
with highest weight $(\mu_1,\mu_2,\mu_3)$ with
$\mu_1 + \mu_2 +\mu_3=0;$
it is actually keeping track of this central character condition (on $\GL_3(\bc)$) that complicates our analysis.

Denote the dimension of the zero weight space in the irreducible representation of
$\GL_4(\bc)$ with highest weight
$(\lambda_1 \geq \lambda_2 \geq   \lambda_3 \geq  \lambda_4)$ by
$d(\lambda_1,\lambda_2,\lambda_3,\lambda_4)$.  Similarly,
denote the dimension of the zero weight space in the irreducible representation of
$\GL_3(\bc)$ with highest weight
$(\mu_1,\mu_2,\mu_3)$ by $d(\mu_1,\mu_2,\mu_3)$;
we will always assume that the central character
of this representation of $\GL_3(\bc)$ is trivial, and so $d(\mu_1,\mu_2,\mu_3)$ is a positive integer, explicitly given by
Lemma 5.1.  We remind the reader from loc. cit. that the value of  $d(\mu_1,\mu_2,\mu_3)$ is a
polynomial in $(\mu_1,\mu_2,\mu_3)$ (of degree 1) which depends on
whether $\mu_2$ is non-negative or non-positive.

Denote the interval $[\lambda_1, \lambda_2]$ by $I_1$ (we abuse the notation
$[\lambda_1, \lambda_2]$ which is customarily denoted by $[\lambda_2, \lambda_1]$),
the interval $[\lambda_2, \lambda_3]$ by $I_2$,
and the interval $[\lambda_3, \lambda_4]$ by $I_3$. Our problem consists in
choosing
integers $\mu_i \in I_i$ such that $\mu_1 + \mu_2 +\mu_3=0$.

One lucky situation is when the set $I_j + I_k \subset -I_\ell$, for a triple $\{i,j,k\} = \{1,2,3\}$,
 in which case, one can choose $\mu_j \in I_j$, $\mu_k \in I_k$ arbitrarily, and
then $\mu_\ell = -(\mu_j + \mu_k)$ automatically belongs to $I_\ell$. This is what happens in
cases I and
II below; but the other cases that we deal with in III, $\dots$, VI, the analysis is considerably more
complicated.

\vspace{4mm}

{\bf Case I:} $\lambda_2 \leq 0$, and therefore $ \lambda_1 \geq 0 \geq \lambda_2 \geq \lambda_3 \geq \lambda_4$.

This implies that $\mu_1 \geq 0 \geq \mu_2 \geq \mu_3$; in particular, in this case $\mu_2$ is always $\leq 0$.
Further, $I_2 + I_3 = [\lambda_2+\lambda_3, \lambda_3+\lambda_4]$ is contained in $-I_1 = [-\lambda_2,-\lambda_1].$

Therefore, by Lemma 5.1,
{\scriptsize
\begin{eqnarray*} d(\lambda_1,\lambda_2,\lambda_3,\lambda_4)
& = & \sum_{\mu_i \in I_i} d(\mu_1,\mu_2,\mu_3) \\
& = & \sum_{\mu_i \in I_i} (\mu_2 - \mu_3 +1)  \\
& = & \sum_{ \begin{array}{c} {\lambda_2 \geq \mu_2 \geq \lambda_3}  \\ {\lambda_3 \geq \mu_3 \geq \lambda_4} \end{array}} (\mu_2 - \mu_3 +1)  \\
& = & \sum_{ \lambda_3 \geq \mu_3 \geq \lambda_4}
\left [\frac{(\lambda_2 + \lambda_3)(\lambda_2-\lambda_3+1) -(\mu_3-1)(\lambda_2 -
\lambda_3 +1)}{2} \right ] \\
& = &
\frac{(\lambda_2 + \lambda_3)(\lambda_2 -\lambda_3+1)(\lambda_3 -\lambda_4+1)}{2} -
\frac{(\lambda_2-\lambda_3+1)(\lambda_3+\lambda_4-2) (\lambda_3-\lambda_4+1)}{2}  \\
& = & \frac{1}{2} (\lambda_2-\lambda_3+1)(\lambda_3-\lambda_4+1) (\lambda_2-\lambda_4+2).
\end{eqnarray*}
}

{\bf Case II:} $\lambda_3 \geq 0$, and therefore $ \lambda_1  \geq \lambda_2 \geq \lambda_3 \geq 0 \geq \lambda_4$.

This implies that $\mu_1  \geq \mu_2  \geq  0$; in particular, in this case $\mu_2$ is always $\geq 0$.
Further, $I_1 + I_2 = [\lambda_1+\lambda_2, \lambda_2+ \lambda_3]$ is contained in $-I_3 = [-\lambda_4,-\lambda_3].$

Therefore, by Lemma 5.1,
\begin{eqnarray*} d(\lambda_1,\lambda_2,\lambda_3,\lambda_4)
& =  & \sum_{\mu_i \in I_i} d(\mu_1,\mu_2,\mu_3)  \\
& = &\sum_{\mu_i \in I_i} (\mu_1 - \mu_2 +1)  \\
& = & \sum_{\begin{array}{c} \lambda_1 \geq \mu_1 \geq \lambda_2 \\ \lambda_2 \geq \mu_2 \geq \lambda_3 \end{array}} (\mu_1 - \mu_2 +1)  \\
& = & \frac{1}{2} (\lambda_1-\lambda_2+1)(\lambda_2-\lambda_3+1) (\lambda_1-\lambda_3+2).
\end{eqnarray*}

{\bf Rest of the  cases:} $\lambda_2 >  0 > \lambda_3 $, and therefore
 $ \lambda_1  \geq \lambda_2 > 0 >  \lambda_3  \geq \lambda_4$.

Given that $\mu_i \in I_i$ with $\mu_1+\mu_2+\mu_3=0$, we find that
$$\lambda_3 \geq -(\mu_1+\mu_2) \geq \lambda_4,$$
and therefore,
$$-\lambda_4 -\mu_2 \geq \mu_1 \geq -\lambda_3-\mu_2.$$
Since we already have
$$\lambda_1 \geq \mu_1 \geq \lambda_2,$$
$\mu_1$ is in the intersection of the two intervals
$-\lambda_4 -\mu_2 \geq \mu_1 \geq -\lambda_3-\mu_2$
and $\lambda_1 \geq \mu_1 \geq \lambda_2.$ Therefore, $\mu_1$ must belong to the
interval
$$ I(\mu_2)= [min (\lambda_1,-\lambda_4-\mu_2), max(-\lambda_3-\mu_2,\lambda_2)].$$

Conversely, it is clear that if $\mu_2 \in I_2$, $\mu_1 \in I(\mu_2),$ and $\mu_3=-(\mu_1+\mu_2)$, then each of the $\mu_i$ belongs to
$I_i$, and $\mu_1+\mu_2+\mu_3=0$.

Thus, we can start the calculation of the dimension of the zero weight space
in the representation of $\GL_4(\bc)$
with highest weight $(\lambda_1 \geq \lambda_2 \geq   \lambda_3 \geq  \lambda_4),$
and with trivial central character as
\begin{eqnarray*} d(\lambda_1,\lambda_2,\lambda_3,\lambda_4)   &= & \sum_{\mu_i \in I_i} d(\mu_1,\mu_2,\mu_3) \\
& = & \sum_{ \begin{array}{c}\lambda_2 \geq \mu_2 > 0 \\
\mu_1 \in I(\mu_2)
\end{array} } (\mu_1 - \mu_2 +1)  +
 \sum_{ \begin{array}{c}0 \geq \mu_2 \geq \lambda_3  \\
\mu_1 \in I(\mu_2)
\end{array}} (\mu_2 - \mu_3 +1).
\end{eqnarray*}
At this point, we assume that $\lambda_1 + \lambda_4 \geq 0$. In this case, if $\mu_2\geq 0$, then
$\lambda_1 \geq -\lambda_4 -\mu_2$. On the other hand, under the same condition (i.e., $\lambda_1+\lambda_4 \geq 0$),
if $\mu_2 \leq 0$, then $-\lambda_3 -\mu_2 \geq \ \lambda_2$. This means that for $\mu_2\geq 0$,
$ I(\mu_2)= [min (\lambda_1,-\lambda_4-\mu_2), max(-\lambda_3-\mu_2,\lambda_2)] =
[-\lambda_4-\mu_2, max(-\lambda_3-\mu_2,\lambda_2)]$, and for $\mu_2 \leq 0$, $I(\mu_2) =
[min (\lambda_1,-\lambda_4-\mu_2), -\lambda_3-\mu_2]$. Therefore, we get
\begin{eqnarray*}  d(\lambda_1,\lambda_2,\lambda_3,\lambda_4)  &= & \sum_{\mu_i \in I_i} d(\mu_1,\mu_2,\mu_3) \\
& = & \sum_{\begin{array}{c} \lambda_2 \geq \mu_2 > 0 \\   -\lambda_4-\mu_2 \geq \mu_1 \geq max(-\lambda_3 -\mu_2, \lambda_2) \end{array} } (\mu_1 - \mu_2 +1)  \\
&+&
 \sum_{ \begin{array}{c} 0 \geq \mu_2 \geq \lambda_3  \\
 min(\lambda_1, -\lambda_4-\mu_2) \geq \mu_1 \geq -\lambda_3 -\mu_2 \end{array} }
 (\mu_1 +2 \mu_2 +1).
\end{eqnarray*}
At this point, we assume that besides $\lambda_1+\lambda_4 \geq 0$, we also have $2\lambda_2+\lambda_3 \geq 0$; this latter condition has the effect that
the region $[\lambda_2,0]$ where $\mu_2$ is supposed to belong, splits into two regions where $max(\lambda_3-\mu_2,\lambda_2)$ takes the two
possible options. Similarly, the region $[0,\lambda_3]$ where $\mu_2$ belongs in the second sum gets divided into two regions.

\vspace{3mm}

{\bf Case III:} $\lambda_1 + \lambda_4 \geq 0$ and $2\lambda_2+\lambda_3 \geq 0$.
{\scriptsize
\begin{eqnarray*}  d(\lambda_1,\lambda_2,\lambda_3,\lambda_4)  & = & \sum_{\mu_i \in I_i} d(\mu_1,\mu_2,\mu_3) \\
& = &
\sum_{\begin{array}{c} \lambda_2 \geq \mu_2 \geq -(\lambda_2+\lambda_3) \\   -\lambda_4-\mu_2 \geq \mu_1 \geq  \lambda_2 \end{array} } (\mu_1 - \mu_2 +1)  +
\!\!\!\sum_{\begin{array}{c} -(\lambda_2 +\lambda_3) > \mu_2 > 0 \\   -\lambda_4-\mu_2 \geq \mu_1 \geq  -\lambda_3 -\mu_2 \end{array} } (\mu_1 - \mu_2 +1)   \\
&+& \sum_{ \begin{array}{c} 0 \geq \mu_2 \geq  -(\lambda_1 + \lambda_4)  \\   -\lambda_4-\mu_2  \geq \mu_1 \geq -\lambda_3 -\mu_2 \end{array} } (\mu_1 + 2\mu_2  +1)
+ \!\!\!\sum_{ \begin{array}{c}   -(\lambda_1 + \lambda_4) > \mu_2 \geq \lambda_3   \\   \lambda_1 \geq \mu_1 \geq -\lambda_3 -\mu_2 \end{array} } (\mu_1 + 2 \mu_2 +1) \\
&=& \frac{1}{4}[-4\lambda_1^3-2\lambda_1^2\lambda_2-2\lambda_2^3-6\lambda_1^2\lambda_3+4\lambda_1\lambda_2\lambda_3-4\lambda_2^2\lambda_3-4\lambda_1\lambda_3^2-2\lambda_3^3-6\lambda_1^2\lambda_4-4\lambda_1\lambda_2\lambda_4\\
& &-8\lambda_1\lambda_3\lambda_4-4\lambda_2\lambda_3\lambda_4-4\lambda_3^2\lambda_4+4\lambda_1\lambda_4^2+4\lambda_2\lambda_4^2+6\lambda_4^3-12\lambda_1^2-5\lambda_1\lambda_2+\lambda_2^2-7\lambda_1\lambda_3\\
& &+8\lambda_2\lambda_3+\lambda_3^2-34\lambda_1\lambda_4-15\lambda_2\lambda_4-13\lambda_3\lambda_4-20\lambda_4^2+5\lambda_1+3\lambda_2+5\lambda_3-\lambda_4+4] \\
&=&
\frac{1}{4}[2\lambda_1^2\lambda_2-2\lambda_1\lambda_2^2-4\lambda_2^3-2\lambda_1^2\lambda_3-10\lambda_2^2\lambda_3-6\lambda_1\lambda_3^2-
6\lambda_2\lambda_3^2-4\lambda_3^3+2\lambda_1^2+4\lambda_1\lambda_2\\
& &-4\lambda_2^2-4\lambda_2\lambda_3-6\lambda_3^2+6\lambda_1+4\lambda_2+6\lambda_3+4],\\
&=&
-\frac{1}{2} (\lambda_1+\lambda_2+2\lambda_3+1)
(-\lambda_1\lambda_2+2\lambda_2^2+\lambda_1\lambda_3+\lambda_2\lambda_3+
\lambda_3^2-\lambda_1+\lambda_3-2),
\end{eqnarray*}
}where in the second last equality, we have used the equation
$\lambda_4=-(\lambda_1+\lambda_2+\lambda_3)$ to write
the polynomial in only $\lambda_1,\lambda_2,\lambda_3$.

\vspace{3mm}

{\bf Case IV:} $\lambda_1 + \lambda_4 \geq 0$ and $2\lambda_2+\lambda_3 <  0$.
{\scriptsize
\begin{eqnarray*}  d(\lambda_1,\lambda_2,\lambda_3,\lambda_4)  & = & \sum_{\mu_i \in I_i} d(\mu_1,\mu_2,\mu_3) \\
& = &
\sum_{\begin{array}{c} \lambda_2 \geq \mu_2 > 0 \\   -\lambda_4-\mu_2 \geq \mu_1 \geq  -\lambda_3-\mu_2 \end{array} }\!\!\! (\mu_1 - \mu_2 +1)  \\
&+& \sum_{ \begin{array}{c} 0 \geq \mu_2 \geq   -(\lambda_1 + \lambda_4)  \\   -\lambda_4-\mu_2  \geq \mu_1 \geq -\lambda_3 -\mu_2 \end{array} } \!\!\!(\mu_1 + 2\mu_2  +1)
+ \sum_{ \begin{array}{c}   -(\lambda_1 + \lambda_4) > \mu_2 \geq \lambda_3   \\   \lambda_1 \geq \mu_1 \geq -\lambda_3 -\mu_2 \end{array} } \!\!\!(\mu_1 + 2 \mu_2 +1) \\
&=&
\frac{1}{4}[-2\lambda_1^3-2\lambda_1^2\lambda_2-2\lambda_1^2\lambda_3+2\lambda_1\lambda_2\lambda_3-4\lambda_2^2\lambda_3-2\lambda_1\lambda_3^2-2\lambda_2\lambda_3^2-2\lambda_3^3-4\lambda_1^2\lambda_4-4\lambda_1\lambda_2\lambda_4\\
& &+4\lambda_2^2\lambda_4-4\lambda_1\lambda_3\lambda_4-2\lambda_2\lambda_3\lambda_4-2\lambda_3^2\lambda_4+2\lambda_2\lambda_4^2-2\lambda_3\lambda_4^2+2\lambda_4^3-5\lambda_1^2-3\lambda_1\lambda_2-4\lambda_2^2+\\
& &3\lambda_2\lambda_3+\lambda_3^2-14\lambda_1\lambda_4-7\lambda_2\lambda_4-7\lambda_4^2+\lambda_1-\lambda_2+\lambda_3-5\lambda_4+4] \\
&=&
\frac{1}{4}[2\lambda_1^2\lambda_2-2\lambda_1\lambda_2^2-4\lambda_2^3-2\lambda_1^2\lambda_3-10\lambda_2^2\lambda_3-6\lambda_1\lambda_3^2-6\lambda_2\lambda_3^2-4\lambda_3^3+2\lambda_1^2+4\lambda_1\lambda_2\\
& &-4\lambda_2^2-4\lambda_2\lambda_3-6\lambda_3^2+6\lambda_1+4\lambda_2+6\lambda_3+4] \\
&=& -\frac{1}{2} (\lambda_1+\lambda_2+2\lambda_3+1)
(-\lambda_1\lambda_2+2\lambda_2^2+\lambda_1\lambda_3+\lambda_2\lambda_3+
\lambda_3^2-\lambda_1+\lambda_3-2),
\end{eqnarray*}
}
where again in the second last equality, we have used the equation
$\lambda_4=-(\lambda_1+\lambda_2+\lambda_3)$ to write
the polynomial  in only $\lambda_1,\lambda_2,\lambda_3$.

\vspace{3mm}

{\bf Case V:} $\lambda_1 + \lambda_4 < 0$ and $\lambda_2+ 2\lambda_3 \leq 0$.
{\scriptsize
\begin{eqnarray*}  d(\lambda_1,\lambda_2,\lambda_3,\lambda_4)  & = & \sum_{\mu_i \in I_i} d(\mu_1,\mu_2,\mu_3) \\
& = &
\sum_{\begin{array}{c} \lambda_2 \geq \mu_2 >  (\lambda_2+\lambda_3) \\   -\lambda_4-\mu_2 \geq \mu_1 \geq  \lambda_2 \end{array} } (\mu_1 - \mu_2 +1)  +
\sum_{\begin{array}{c} (\lambda_2 +\lambda_3) \geq  \mu_2 \geq 0 \\   \lambda_1 \geq \mu_1 \geq  \lambda_2 \end{array} } (\mu_1 - \mu_2 +1)   \\
&+& \sum_{ \begin{array}{c} 0 > \mu_2 >  (\lambda_1 + \lambda_4)  \\   \lambda_1 \geq \mu_1 \geq \lambda_2 \end{array} } (\mu_1 + 2\mu_2  +1)
+ \sum_{ \begin{array}{c}   \lambda_1 + \lambda_4  \geq \mu_2 \geq \lambda_3   \\   \lambda_1 \geq \mu_1 \geq -\lambda_3 -\mu_2 \end{array} } (\mu_1 + 2 \mu_2 +1) \\
&=&
\frac{1}{4}[-2\lambda_1^3+2\lambda_1^2\lambda_2+2\lambda_1\lambda_2^2+2\lambda_2^3-2\lambda_1^2\lambda_3+2\lambda_1\lambda_2\lambda_3+2\lambda_2^2\lambda_3-4\lambda_1\lambda_3^2+4\lambda_2\lambda_3^2+4\lambda_1\lambda_2\lambda_4\\
& &+2\lambda_2^2\lambda_4+4\lambda_1\lambda_3\lambda_4-2\lambda_2\lambda_3\lambda_4+4\lambda_1\lambda_4^2+2\lambda_2\lambda_4^2+2\lambda_3\lambda_4^2+2\lambda_4^3-7\lambda_1^2+\lambda_2^2-7\lambda_1\lambda_3\\
& &+3\lambda_2\lambda_3-4\lambda_3^2-14\lambda_1\lambda_4-3\lambda_3\lambda_4-5\lambda_4^2+5\lambda_1-\lambda_2+\lambda_3-\lambda_4+4] \\
&=&
\frac{1}{4}[2\lambda_1^2\lambda_2-2\lambda_1\lambda_2^2-2\lambda_1^2\lambda_3+2\lambda_2^2\lambda_3-6\lambda_1\lambda_3^2+6\lambda_2\lambda_3^2+2\lambda_1^2+4\lambda_1\lambda_2-4\lambda_2^2\\
& &-4\lambda_2\lambda_3-6\lambda_3^2+6\lambda_1+2\lambda_3+4] \\
&= & \frac{1}{2} (-\lambda_1+\lambda_2-1)(-\lambda_1\lambda_2+\lambda_1\lambda_3+\lambda_2\lambda_3+3\lambda_3^2-
\lambda_1-2\lambda_2-\lambda_3-2).
\end{eqnarray*}
}

{\bf Case VI:}  $\lambda_1 + \lambda_4 < 0$ and $\lambda_2+ 2\lambda_3 \geq 0$.
{\scriptsize
\begin{eqnarray*}  d(\lambda_1,\lambda_2,\lambda_3,\lambda_4)  & = & \sum_{\mu_i \in I_i} d(\mu_1,\mu_2,\mu_3) \\
& = &
\sum_{\begin{array}{c} \lambda_2 \geq \mu_2 > (\lambda_2+\lambda_3) \\   -\lambda_4-\mu_2 \geq \mu_1 \geq  \lambda_2 \end{array} } (\mu_1 - \mu_2 +1)  +
\sum_{\begin{array}{c} (\lambda_2 +\lambda_3) \geq  \mu_2 \geq 0 \\   \lambda_1 \geq \mu_1 \geq  \lambda_2 \end{array} } (\mu_1 - \mu_2 +1)   \\
&+& \sum_{ \begin{array}{c} 0 >  \mu_2  \geq \lambda_3  \\   \lambda_1 \geq \mu_1 \geq \lambda_2 \end{array} } (\mu_1 + 2\mu_2  +1)  \\
& =& \frac{1}{4}
[-2\lambda_1^3+2\lambda_1^2\lambda_2+2\lambda_1\lambda_2^2+2\lambda_2^3-2\lambda_1^2\lambda_3+2\lambda_1\lambda_2\lambda_3+2\lambda_2^2\lambda_3-4\lambda_1\lambda_3^2+4\lambda_2\lambda_3^2+4\lambda_1\lambda_2\lambda_4\\
& &+2\lambda_2^2\lambda_4+4\lambda_1\lambda_3\lambda_4-2\lambda_2\lambda_3\lambda_4+4\lambda_1\lambda_4^2+2\lambda_2\lambda_4^2+2\lambda_3\lambda_4^2+2\lambda_4^3-7\lambda_1^2+\lambda_2^2-7\lambda_1\lambda_3\\
& &+3\lambda_2\lambda_3-4\lambda_3^2-14\lambda_1\lambda_4-3\lambda_3\lambda_4-5\lambda_4^2+5\lambda_1-\lambda_2+\lambda_3-\lambda_4+4] \\
&=&
\frac{1}{4}[2\lambda_1^2\lambda_2-2\lambda_1\lambda_2^2-2\lambda_1^2\lambda_3+2\lambda_2^2\lambda_3-6\lambda_1\lambda_3^2+6\lambda_2\lambda_3^2+2\lambda_1^2+4\lambda_1\lambda_2-4\lambda_2^2\\
& &-4\lambda_2\lambda_3-6\lambda_3^2+6\lambda_1+2\lambda_3+4] \\
&= & \frac{1}{2} (-\lambda_1+\lambda_2-1)(-\lambda_1\lambda_2+\lambda_1\lambda_3+\lambda_2\lambda_3+3\lambda_3^2-
\lambda_1-2\lambda_2-\lambda_3-2).
\end{eqnarray*}
}
\end{proof}
\begin{remark} (a)
The fact that the answers in the cases III and IV above are the same (similarly in
the cases V and VI) seems not obvious apriori
  before the final answer is calculated via a software.

  (b) It is curious to note that the polynomials $p_1$ and $p_3$ are equal for $\lambda_2=0$, and the
polynomials $p_2$ and $p_3$ are equal for $\lambda_3=0$, but  the polynomials
$p_1$ and $p_3$ are not equal for $\lambda_3=0$.
This means that $d(\lambda_1,\lambda_2,\lambda_3,\lambda_4)$ which is a polynomial function in the interior of a polyhedral cone
is not always given by the same polynomial on the boundary.

\end{remark}

\vspace{2cm}

S.K.: Department  of Mathematics, University of North Carolina, Chapel Hill, NC 27599, USA (email: shrawan$@$email.unc.edu).

\vspace{3mm}

D.P.: School of Mathematics, Tata Institute of Fundamental Research, Colaba, Mumbai, 400005,  INDIA (email: dprasad$@$math.tifr.res.in).
\end{document}